\documentclass{amsart}
\usepackage[latin1]{inputenc}
\usepackage{amsmath}
\usepackage{amsfonts}
\usepackage{amssymb}
\usepackage{graphicx}
\usepackage{graphics} 
\usepackage{lscape}
\usepackage{comment}
\usepackage{mathrsfs}

\newcommand{\QI}{q}
\usepackage{amsmath}
\numberwithin{equation}{section}

\numberwithin{theorem}{section}
\newtheorem{lemma}{Lemma}
\numberwithin{lemma}{section}
\newtheorem{prop}{Proposition}
\numberwithin{prop}{section}
\newtheorem{corol}{Corollary}
\numberwithin{corol}{section}

\numberwithin{remark}{section}

\numberwithin{defi}{section}

\numberwithin{exe}{section}

\title[Processor Sharing Queue with batches (II)]{\textbf{Sojourn time in a
$M^{[X]}/M/1$ Processor Sharing Queue with batch arrivals (II)}}
\author{F. Guillemin*, V.K. Quintuna Rodriguez*, A.Simonian **, R.Nasri **
\\ 
Orange Labs}
\address{Postal addresses: * Orange Labs Networks Lannion, 
2 avenue Pierre Marzin
22307 Lannion Cedex, France /  
** Orange Labs, OLN/GDM, Orange Gardens, 
44 avenue de la République, CS 50010, 92326 Châtillon Cedex, 
France}
\email{[fabrice.guillemin,alain.simonian,ridha.nasri]@orange.com}

\begin{document}

\date{Version of \today}

\begin{abstract}
For the $M^{[X]}/M/1$ processor Sharing queue with batch arrivals, the sojourn time $\Omega$ of a batch is investigated. 

We first show that the distribution of $\Omega$ can be generally obtained from an infinite linear differential system. When further assuming that the batch size has a geometric distribution with given parameter $\QI \in [0,1[$, this differential system is further analyzed by means of an associated bivariate generating function $(x,u,v) \mapsto E(x,u,v)$. Specifically, denoting by  $s \mapsto E^*(s,u,v)$ the one-sided Laplace transform of $E(\cdot,u,v)$ and defining
$$
\Phi(s,u,v) = P(s,u) \, (1-v) \, F^*(s,u,uv), 
\quad 0 < \vert u \vert < 1, \, \vert v \vert < 1, 
$$
for some known polynomial $P(s,u)$ and where
$$
F^*(s,u,v) = \frac{E^*(s,u,v)-E^*(s,\QI,v)}{u-\QI},
$$
we show that the function $\Phi$ verifies an inhomogeneous linear partial differential equation (PDE) 
$$
\frac{\partial \Phi}{\partial u} -  
\left [ \frac{u - \QI}{P(s,u)} \right ] v(1-v) \, 
\frac{\partial \Phi}{\partial v} + \ell(s,u,v) = 0
$$
for given $s$, where the last term $\ell(s,u,v)$ involves both $E^*(s,\QI,v)$ and the first order derivative $\partial E^*(s,\QI,v)/\partial v$ at the boundary point $u = \QI$. Solving this PDE for $\Phi$ via its characteristic curves and with the required analyticity properties eventually determines the one-sided Laplace transform $E^*$.

By means of a Laplace inversion of this transform $E^*$, the distribution function of the sojourn time $\Omega$ of a batch is then given in an integral form. The tail behavior of the distribution of sojourn time $\Omega$ is finally derived.

\end{abstract}

\maketitle


\section{Introduction}
\label{Sec:Intro}


\subsection{The queuing model}

Among its potential benefits, the introduction of Cloud Computing in network and service systems permits the so-called ``virtualization'', whereby the treatment of a single request is broken into several components (``jobs'') whose service is performed on banalized (``virtual'') service machines. Meanwhile, assessing the performance of these virtual architectures is necessary. In fact, requests incoming such systems are delivered in batches and to meet constraints on the completion time of parallelized jobs composing each  individual batch is mandatory for the system design and dimensioning. 

From the modeling approach developed in \cite{GuiSim18, GuiSim19}, evaluating the performance of such systems can be envisaged as follows. A single request generates a batch of several jobs to be executed in parallel on a unique server (this server represents here the sum of individual capacities of processing units composing the Cloud). In view of the random nature of the flow of requests in time, the probability distribution of the sojourn time $W$ of a single job and of the sojourn time $\Omega$ of a whole batch both describe the performance of this system in stationary conditions. These distributions can then be used in order to guarantee that, with a large probability, the service of a job or a batch is completed before some finite time lag. The distribution of the job sojourn time $W$ has been fully determined in \cite{GuiSim18}; the present paper now addresses the derivation of the distribution of the batch sojourn time $\Omega$.

Following the queuing model considered in \cite{GuiSim18}, the server is represented by a single queue fed by the incoming flow of requests; in the present model, we assume that this flow is Poisson with constant rate 
$\lambda$. Any incoming request simultaneously brings a batch of jobs for service, with the batch size (in terms of number of jobs) denoted by $B$; the service time of any job pertaining to this batch is denoted by $S$. All random variables $B$ (resp. $S$) associated with consecutive batches 
(resp. with jobs contained in a batch) are supposed to be mutually independent and identically distributed. In view of a fair treatment of requests by the server, we finally assume that all jobs in the queue are served according to the Processor-Sharing (PS) discipline. 

Let $\mathbb{P}(B = b) = q_b$, $b\geq 1$, define the distribution of the size 
$B$ of any batch. Assuming $\mathbb{E}(B) < +\infty$ and that the service $S$ is exponentially distributed with parameter $\mu$, the corresponding 
$M^{[X]}/M/1$ queue has a stationary regime provided that the stability condition 
\begin{equation}
\varrho = \frac{\lambda \, \mathbb{E}(B)}{\mu} < 1
\label{stab0}
\end{equation}
holds (\cite{KLEI75}, Vol.I, §4.5). As mentioned above, the sojourn time 
$W$ of a single job has been already addressed in \cite{GuiSim18}; specifically, the distribution function of $W$ has been given an integral representation in the case when the distribution $(q_b)_{b \geq 1}$ of the batch size is geometric; this has further enabled the derivation of asymptotics for the distribution tail together with convergence results under heavy load condition. 

As motivated above, the present paper now aims at characterizing the sojourn time $\Omega$ of a whole batch incoming the $M^{[X]}/M/1$ queue with PS discipline; both the stationary distribution and its tail behavior at infinity will be explicitly derived within the same assumption for the distribution of the batch size. To our knowledge, the distribution of sojourn time $\Omega$ for a batch size $B > 1$ has not been addressed so far in the literature.

\subsection{Contribution of the paper}
For the considered $M^{[X]}/M/1$ PS queue, 

\textbf{A)} we first show that the conditional distribution functions 
$E_{n,b}$ of sojourn time $\Omega$, given the job occupancy $n \geq 0$ at the batch arrival instant and that this batch contains $b$ jobs, verify an infinite-dimensional linear differential system (Section \ref{Sec:Prob}). The  resolution of such an infinite system cannot, however, be generally performed for any batch size distribution;

\textbf{B)} we further assume that the batch size is geometrically distributed with some fixed parameter $\QI \in [0,1[$. Defining the bivariate generating function 
$E$ by
$$
E(x,u,v) = \sum_{n \geq 0} \sum_{b \geq 1} E_{n,b}(x)u^n v^b, 
\qquad x \in \mathbb{R}^+, \; \vert u \vert < 1, \; \vert v \vert < 1,
$$
the resolution of this differential system is then reduced to that of a so-called ``governing''Partial Differential Equation (PDE) verified by this  unknown function $E$. This governing equation is linear and of second order but 
non standard in that it also involves the unknown boundary values of $E$ at point $u = \QI$ (Section \ref{Sec:PDE}). In order to solve 
this governing PDE for $E$, we consider the one-sided Laplace transform 
$E^*(\cdot,u,v)$ of $E(\cdot,u,v)$ defined by
$$
E^*(s,u,v) = \int_0^{+\infty} E(x,u,v) e^{-s x} \mathrm{d}x, 
\quad s \geq 0,
$$
for a given pair $(u,v)$ with $\vert u \vert < 1$, $\vert v \vert < 1$.  Introducing the successive function changes $E^* \mapsto F^* \mapsto \Phi$ where   
$$
F^*(s,u,v) = \frac{E^*(s,u,v)-E^*(s,\QI,v)}{u-\QI} 
$$
and
$$
\Phi(s,u,v) = P(s,u) \, (1-v) \, F^*(s,u,uv)
$$
for some known quadratic polynomial $P(s,u)$ in variable $u$, 
the governing PDE for $E$ is then shown to translate into a first order linear PDE for function $\Phi$, namely
\begin{equation}
\frac{\partial \Phi}{\partial u} -  
\left [ \frac{(u - \QI)}{P(s,u)} \right ] v(1-v) \, 
\frac{\partial \Phi}{\partial v} + \ell(s,u,v) = 0
\label{EDPPhi}
\end{equation}
where the last term $\ell(s,u,v)$ in (\ref{EDPPhi}) involves both 
$E^*(s,\QI,v)$ and the first order derivative 
$\partial E^*(s,\QI,v)/\partial v$ at the boundary point $u = \QI$. Solving equation (\ref{EDPPhi}) via its characteristic curves and with the required analyticity properties eventually determines the Laplace transform 
$E^*$ as $E^*(s,u,v) = E^*(s,\QI,v) + (u-\QI)F^*(s,u,v)$ with the integral representation
\begin{equation}
F^*(s,u,v) = \frac{u}{(u-v)P(s,u)} \, 
\Phi \left ( s, u, \frac{v}{u} \right )
\label{SolF*}
\end{equation}
where
$$
\Phi(s,u,v) = \int_u^{U^-(s)} \frac{1-v}{1-v + v \, \mathfrak{R}(s,u;\zeta)} \cdot L \left(s,\zeta, \zeta \, 
\frac{v \, \mathfrak{R}(s,u;\zeta)}{1-v + v \, \mathfrak{R}(s,u;\zeta)} \right) 
\, \dfrac{\mathrm{d}\zeta}{\zeta}
$$
with
$$
\mathfrak{R}(s,u;\zeta) = 
\left ( \frac{\zeta-U^-(s)}{u-U^-(s)} \right )^{C^-(s)-1}
\left ( \frac{\zeta-U^+(s)}{u-U^+(s)} \right )^{C^+(s)-1},
$$
$U^-(s)$, $U^+(s)$ denoting the two roots of quadratic polynomial 
$P(s,u)$;  finally, the function $L$ involved in the integrand of formula 
(\ref{SolF*}) is given by 
$$
L(s,u,v) = \frac{v(1-uv)}{(1-u)^2(1-v)^2} + 
(u+v)E^*(s,\QI,v) + v(v-s-1-\varrho) 
\frac{\partial E^*}{\partial v}(s,\QI,v).
$$

To further determine the auxiliary function $L$ involving the unknown function $E^*$ on the boundary line $u = \QI$, it is shown that $L$ must solve the integral equation
\begin{equation}
\int_0^{U^-(s)} \frac{1-v}{1-v + v \, \mathfrak{R}(s,0;\zeta)} \cdot 
L \left(s,\zeta, \zeta \, 
\frac{v \, \mathfrak{R}(s,0;\zeta)}{1-v + v \, \mathfrak{R}(s,0;\zeta)} \right) 
\, \dfrac{\mathrm{d}\zeta}{\zeta} = 0
\label{INTL}
\end{equation}
for all $v \in \mathbb{D}$. This integral equation is in turn non-standard as both the ``external'' variable $v$ and the integration variable $\zeta$ are involved in the arguments of $L$.


\section{A general differential system}
\label{Sec:Prob}


In this section, we establish that the distribution function of the batch sojourn time $\Omega$ can be derived from the solution of an infinite linear differential system. In the rest of this paper, the service rate $\mu$ will be normalized to 1, so that the arrival rate $\lambda$ is set to $\varrho$ with 
$\varrho \mathbb{E}(B) < 1$ according to condition (\ref{stab0}).

Given a batch size $B = b$, $b \geq 1$, the sojourn time $\Omega$ equals by definition the maximum
\begin{equation}
\Omega = \max_{1 \leq k \leq b} W_k
\label{defWbar}
\end{equation}
of the sojourn times $W_k$, $1 \leq k \leq b$, of jobs which build up this batch. We will denote by $\Omega_{n,b}$ the sojourn time of a batch in the queue, given that 

$\bullet$ $n \geq 0$ jobs are already present in that queue at its arrival instant 

$\bullet$ and this batch has size $b \geq 1$.

\noindent
For given $n \geq 0$, $b \geq 1$, we denote by $E_{n,b}$ the complementary cumulative distribution function of sojourn time $\Omega_{n,b}$, that is,
$$
E_{n,b}(x) = \mathbb{P}(\Omega_{n,b} > x), \qquad 
x \in \mathbb{R}^+.
$$
As $\Omega_{n,b} > 0$ almost surely (since the sojourn time includes the non-zero service times of jobs), we note that 
\begin{equation}
E_{n,b}(0) = 1, \quad n \geq 0, \; b \geq 1.
\label{PW1}
\end{equation}

\begin{prop}
\textbf{The set of distribution functions $E_{n,b}$, $n \geq 0$, $b \geq 1$, verifies the differential system}
\begin{align}
\frac{\mathrm{d} E_{n,b}}{\mathrm{d}x}(x) = & \; 
\varrho \sum_{m \geq 1} q_m E_{n+m,b}(x) \; - 
\nonumber \\
& \; (1+\varrho) E_{n,b}(x) + \frac{n}{n+b} E_{n-1,b}(x) + 
\frac{b}{n+b}E_{n,b-1}(x)
\label{EQU0_Batch}
\end{align}
\textbf{for all $x \in \mathbb{R}^+$, $n \geq 0$ and $b \geq 1$ (by convention, we set $E_{n,b} = 0$ for either index $n < 0$ or $b < 1$).}
\label{DIFF_SYST_Batch}
\end{prop}

\begin{proof}
Consider a tagged batch labeled $\mathfrak{B}$, arriving at some initial time when the system contains $N = n \geq 0$ jobs in the queue, and with size 
$b \geq 1$. Variable $\Omega_{n,b}$ then equals
\begin{equation}
\Omega_{n,b} = 
\left\{
\begin{array}{ll}
X_{1+\varrho} + \Omega_{n,b-1} \quad \quad \quad \; \; \;  
\mathrm{with \; probability} \quad 
\displaystyle \frac{1}{1+\varrho} \times \frac{b}{n+b},
\\ \\
X_{1+\varrho} + \Omega_{n-1,b} \quad \quad \quad \; \; \; \mathrm{with \; probability} \quad 
\displaystyle   \frac{1}{1+\varrho} \times \frac{n}{n+b},
\\ \\
X_{1+\varrho} + \Omega_{n+m,b} \quad \quad \quad \; \; 
\mathrm{with \; probability} \quad 
\displaystyle \frac{\varrho}{1+\varrho} \times q_m, \; \; m \geq 1.
\end{array} 
\right.
\label{OmegaNB}
\end{equation}
where all equalities in (\ref{OmegaNB}) are meant in distribution and with  
$X_{1+\varrho}$ denoting any positive random variable with exponential distribution of parameter $1 + \varrho$. To prove equalities 
(\ref{OmegaNB}), observe that after the arrival time of batch $\mathfrak{B}$, the next event to occur can be either 

\textbf{(i)} a departure due to the service completion of some job in queue 
(with probability $\mu/(\lambda + \mu) = 1/(1 + \varrho)$). In this first case, 
   \begin{itemize}
	 \item[-] the probability that the service of a job pertaining to batch 
$\mathfrak{B}$ is completed is equal to $b/(n+b)$ (since $n$ jobs were present at the arrival time of $\mathfrak{B}$, which has brought a total number of $b$ 
jobs), hence $\Omega_{n,b} = X_{1+\varrho} + \Omega_{n,b-1}$;  
 
	 \item[-] the probability that this service completion does not occur for any job pertaining to batch $\mathfrak{B}$ equals $n/(n+b)$ and we have 
$\Omega_{n,b} = X_{1+\varrho} + \Omega_{n-1,b}$ since another job 
(not pertaining to batch $\mathfrak{B}$) has meanwhile left the queue;
   \end{itemize}

\textbf{(ii)} or the arrival of new batch (with probability 
$\lambda/(\lambda + \mu) = \varrho/(1 + \varrho)$) with some size $m \geq 1$ 
(with probability $q_m$). In this case, the corresponding sojourn time of batch 
$\mathfrak{B}$ equals $\Omega_{n,b} = X_{1+\varrho} + \Omega_{n+m,b}$, due to the memory-less property for the service times of all jobs in the tagged batch 
$\mathfrak{B}$. 

Items \textbf{(i)} and \textbf{(ii)} consequently justify equalities 
(\ref{OmegaNB}) in distribution. Using (\ref{OmegaNB}), we then derive that the Laplace transform 
$e^*_{n,b}:s > 0 \mapsto \mathbb{E}(e^{-s \Omega_{n,b}})$ of sojourn time 
$\Omega_{n,b}$ verifies
\begin{align}
e^*_{n,b}(s) = \; & \frac{b}{(n+b)(s+\varrho+1)}e^*_{n,b-1}(s) + 
\frac{1}{s+\varrho+1}\frac{n}{n+b} \, e^*_{n-1,b}(s) \; +
\nonumber \\
& \; \frac{\varrho}{s + \varrho + 1} 
\sum_{m \geq 1} q_m \, e^*_{n+m,b}(s), \qquad s > 0.
\label{E*}
\end{align}
If $E^*_{n,b}$ now denotes the Laplace transform of the complementary distribution function $E_{n,b}:x \mapsto \mathbb{P}(\Omega_{n,b} > x)$, 
$e^*_{n,b}$ and $E^*_{n,b}$ are related by 
$e^*_{n,b}(s) = 1 - s \, E^*_{n,b}(s)$ for $s > 0$; identity (\ref{E*}) can then be equivalently written in terms of transform 
$E^*_{n,b}$ as
\begin{align}
1 - (s+\varrho+1)E^*_{n,b}(s) = & \; - \frac{b}{n+b}E^*_{n,b-1}(s) 
- \frac{n}{n+b} \, E^*_{n-1,b}(s)
\nonumber \\
& \; - \varrho \sum_{m \geq 1} q_m \, E^*_{n+m,b}(s)
\label{LaplInv0_Batch}
\end{align}
for $s > 0$. Inverting relation (\ref{LaplInv0_Batch}) with respect to the Laplace transformation (noting in the left-hand side that the Laplace inverse of $s \mapsto 1 - s \, E^*_{n,b}(s)$ is the derivative 
$-\mathrm{d} E_{n,b}/\mathrm{d}x$), differential equation 
(\ref{EQU0_Batch}) follows.
\end{proof}

An explicit solution to the infinite system (\ref{EQU0_Batch}) does not seem affordable for any distribution $(q_b)_{b \geq 1}$. In the next section, an alternative formulation to system (\ref{EQU0_Batch}) will be provided in the case of a specific distribution of the batch size. 


\section{Geometric distribution of the batch size}
\label{Sec:PDE}


In the rest of this paper, the distribution of the batch size $B$ will be  assumed to be geometric with given parameter $\QI \in [0,1[$, that is,
\begin{equation}
\forall \; b \geq 1, \quad q_b = (1-\QI) \QI^{b-1}.
\label{defGeo}
\end{equation}
The geometric distribution (\ref{defGeo}) entails, in particular, that
$\mathbb{E}(B) = 1/(1-\QI)$ so that stability condition (\ref{stab0}) now specifies into
\begin{equation}
\varrho < 1 - \QI.
\label{stab1}
\end{equation}
For a geometric batch size distribution, we will show that the resolution of system (\ref{EQU0_Batch}) translates to solving a partial differential equation (PDE) for a generating function associated with distribution functions $E_{n,b}$, $n \geq 0$, $b \geq 1$. 

Specifically, let 
$\mathbb{D} = \{u \in \mathbb{C}, \, \vert u \vert < 1\}$  denote the unit disk in the complex plane. Define the generating functions 
$E_b$, $b \geq 1$, by
\begin{equation}
E_b(x,u) = \sum_{n \geq 0} E_{n,b}(x) u^n, \qquad x \in \mathbb{R}^+, 
\; u \in \mathbb{D},
\label{defEb}
\end{equation}
and the bivariate generating function $E$ by
\begin{equation}
E(x,u,v) = 
\sum_{b \geq 1} E_{b}(x,u) v^b, 
\qquad x \in \mathbb{R}^+, \; (u, v) \in \mathbb{D} \times \mathbb{D}.
\label{defE}
\end{equation}
Note, by definition, that the function $E$ verifies the boundary condition
\begin{equation}
E(x,u,0) = 0, \qquad x \in \mathbb{R}^+, \; u \in \mathbb{D},
\label{defEbc}
\end{equation}
on the line $v = 0$.

\subsection{The governing PDE for $E$}
We can now establish that the differential system (\ref{EQU0_Batch}) translates into the following second order linear PDE for the generating function $E$.

\begin{prop}
\textbf{If the distribution of the batch size is geometric with parameter 
$\QI \in [0,1[$, the generating function $E$ verifies the linear second order partial differential equation}
\begin{align}
u \, \frac{\partial^2 E}{\partial x \partial u}(x,u,v) & + 
v \, \frac{\partial^2 E}{\partial x\partial v}(x,u,v) 
+ \frac{u(u-1)(\varrho + \QI - u)}{u-\QI} 
\frac{\partial E}{\partial u}(x,u,v) 
\nonumber \\
& + v(1 + \varrho - v) \frac{\partial E}{\partial v}(x,u,v) - 
\frac{\varrho (1 - \QI)v}{u-\QI}
\left [ \frac{\partial E}{\partial v}(x,u,v) - 
\frac{\partial E}{\partial v}(x,\QI,v) \right ] 
\nonumber \\
& - (u + v)E(x,u,v) + 
\frac{\varrho (1 - \QI)u}{(u-\QI)^2}(E(x,u,v) - E(x,\QI,v)) = 0
\label{EQU1_BatchE}
\end{align}
\textbf{for $x \in \mathbb{R}^+$ and $(u,v) \in \mathbb{D}^2$}.
\label{PDE1_Batch}
\end{prop}

\noindent
We refer to Appendix \ref{App0} for the proof of Proposition \ref{PDE1_Batch}. Beside its second order, we note that the governing equation 
(\ref{EQU1_BatchE}) involves boundary terms at $u = \QI$ for both $E$ and its  first derivative $\partial E/\partial v$. 

\subsection{A first order PDE for the Laplace transform}
\label{PDEL}
Consider the one-sided Laplace transform $E^*(\cdot,u,v)$ with respect to variable $x \in \mathbb{R}^+$, that is, 
\begin{equation}
E^*(s,u,v) = \int_0^{+\infty} E(x,u,v) e^{-s x} \mathrm{d}x,   
\quad \Re(s) > 0,
\label{deF*Batch}
\end{equation}
for given $(u,v) \in \mathbb{D}^2$; note that definition (\ref{defE}) readily entails the upper bound
$$
\vert E(x,u,v) \vert \leq 
\sum_{n \geq 0} \sum_{b \geq 1} \vert u \vert^n \vert v \vert^b = 
\frac{\vert v \vert}{1 - \vert u \, v \vert}
$$
for all $x \in \mathbb{R}^+$ and given $(u, v) \in \mathbb{D}^2$, which ensures that $E^*$ is analytic in the product 
$\{s \in \mathbb{C}, \; \Re(s) > 0\} \times \mathbb{D}^2$. 
We will now prove that the Laplace transformation $E \mapsto E^*$  translates the second-order governing equation (\ref{EQU1_BatchE}) in variables $x$, $u$, $v$ into a first order linear equation in variables $u$ and $v$ only. 

In this aim, first introduce the quadratic polynomial 
\begin{equation}
P(s,u) = u^2 - (s + 1 + \varrho + \QI)u + s \QI + \varrho + \QI
\label{defPTheta}
\end{equation}
in variable $u$. Recall (\cite{GuiSim18}, Section 4.1) that $P(s,\cdot)$ has two roots $U^-(s)$ and $U^+(s)$ 
given by
\begin{equation}
U^\pm(s) = \frac{s+1+\varrho+\QI \pm \sqrt{\Delta(s)}}{2}
\label{defU+-Theta}
\end{equation}
with $\Delta(s) = s^2 + 2(1+\varrho-\QI)s + (1-\varrho-\QI)^2$; furthermore, roots $U^\pm(s)$ verify the inequalities (\cite{GuiSim18}, Proof of Proposition 4.1, Equ.(4.15))
\begin{equation}
\forall \; s > 0, \qquad q < U^-(s) < 1 < U^+(s).
\label{INEQ-U}
\end{equation}
Besides, we consider the function change $E^* \mapsto F^*$ where $F^*$ is defined by
\begin{equation}
F^*(s,u,v) = 
\left\{
\begin{array}{ll}
\displaystyle \frac{E^*(s,u,v) - E^*(s,\QI,v)}{u-\QI}, \quad s \geq 0, \; u \in \mathbb{D} \setminus \{\QI\}, \; v \in \mathbb{D},
\\ \\
\displaystyle \displaystyle \frac{\partial E^*}{\partial u}(s,\QI,v), \quad \quad \quad \quad \quad \quad
s \geq 0, \; u = \QI,  \; v \in \mathbb{D}.
\end{array} \right.
\label{defFTheta}
\end{equation}
From the latter definition, function $F^*$ is clearly analytic in 
$\{s \; \vert \; s > 0\} \times \mathbb{D}^2$ and it is obviously equivalent to determine either function $E^*$ or $F^*$. Following the boundary condition 
(\ref{defEbc}) verified by $E$, we readily have $E^*(s,u,0) = 0$ for 
$\Re(s) > 0$ and $u \in \mathbb{D}$; definition (\ref{defFTheta}) then entails that $F^*$ verifies the same boundary condition
\begin{equation}
F^*(s,u,0) = 0, \qquad \Re(s) > 0, \; u \in \mathbb{D},
\label{defFbc}
\end{equation}
on the line $v = 0$. 

As detailed below, it proves that $E^*$ verifies a linear PDE 
whose coefficients, however, exhibit polar singularities at point $u = q$. By means of the function change $E^* \mapsto F^*$ introduced in 
(\ref{defFTheta}), such singularities conveniently cancel out when translating this PDE to the new function $F^*$. This can be stated as follows.

\begin{corol}
\textbf{The function $F^*$ defined in (\ref{defFTheta}) verifies the linear partial differential equation}
\begin{align}
& \; u \, P(s,u) \cdot \frac{\partial F^*}{\partial u} + 
v \left [ (u-\QI)(v-s-1-\varrho) + \varrho(1-\QI) \right ] \cdot  
\frac{\partial F^*}{\partial v} 
\nonumber \\ 
& + \; \left [ u(u-s-1-\varrho) + (u-\QI)(u+v) \right ] \cdot F^*
\nonumber \\ 
& + \; L(s,u,v) = 0 
\label{PDEF0}
\end{align}
\textbf{with polynomial $P$ introduced in (\ref{defPTheta}), $F^*$ and its derivatives taken at any point $(s,u,v)$ and}
\begin{align}
L(s,u,v) = & \; \frac{v(1-uv)}{(1-u)^2(1-v)^2} \; + 
\nonumber \\
& \; (u+v) \, E^*(s,\QI,v) + v \, (v-s-1-\varrho) \, 
\frac{\partial E^*}{\partial v}(s,\QI,v)
\label{defL}
\end{align}
\textbf{for $s > 0$ and $(u,v) \in \mathbb{D}^2$}.
\label{Corollary1}
\end{corol}

\begin{proof}
Following condition (\ref{PW1}) for $x = 0$, we first note that the power series $E(0,u,v) = \sum_{n \geq 0, b \geq 1} u^n v^b$ readily sums to
\begin{equation}
E(0,u,v) = \frac{v}{(1-u)(1-v)}, \quad (u,v) \in \mathbb{D}^2;
\label{E0uv}
\end{equation}
besides, after definition (\ref{deF*Batch}) of $E^*$, the Laplace transform of the first derivative $\partial E/\partial x$ is the function 
$s > 0 \mapsto sE^*(s,u,v) - E(0,u,v)$. Taking the Laplace transform of each side of equation (\ref{EQU1_BatchE}) for given $(u,v) \in \mathbb{D}^2$, we then obtain
\begin{align}
& u \, \frac{\partial}{\partial u} \left [ s E^*(s,u,v) - E(0,u,v) \right ] + 
v \, \frac{\partial}{\partial v} \left [ s E^*(s,u,v) - E(0,u,v) \right ] \; +
\nonumber \\
& \frac{u(u-1)(\varrho + \QI - u)}{u-\QI} 
\frac{\partial E^*}{\partial u}(s,u,v) + 
v(1 + \varrho-v) \frac{\partial E^*}{\partial v}(s,u,v) \; - 
\nonumber \\
& \frac{\varrho (1 - \QI)v}{u-\QI}
\left [ \frac{\partial E^*}{\partial v}(s,u,v) - 
\frac{\partial E^*}{\partial v}(s,\QI,v) \right ] - (u + v)E^*(s,u,v) \; + 
\nonumber \\
& \frac{\varrho (1 - \QI)u}{(u-\QI)^2}(E^*(s,u,v) - E^*(s,\QI,v)) = 0
\label{PDEF1}
\end{align}
for all $s > 0$; assembling all factors multiplying the derivative 
$\partial E^*(s,u,v)/\partial u$, the coefficient of this derivative in 
(\ref{PDEF1}) eventually equals
$$
us + \frac{u(u-1)(\varrho + \QI - u)}{u-\QI} = - \frac{u \, P(s,u)}{u-\QI}
$$
where $P(s,u)$ is the polynomial introduced in (\ref{defPTheta}). Reducing all algebraic factors and using expression (\ref{E0uv}) for $E(0,u,v)$, equality 
(\ref{PDEF1}) then equivalently reads
\begin{align}
- \frac{u \, P(s,u)}{u-\QI} 
\cdot & \frac{\partial E^*}{\partial u}(s,u,v) +  
v \left (s + 1 + \varrho - v \right ) \cdot 
\frac{\partial E^*}{\partial v}(s,u,v) - (u + v)E^*(s,u,v) \; =  
\nonumber \\
& \frac{v(1-uv)}{(1-u)^2(1-v)^2} + \frac{\varrho (1 - \QI)v}{u-\QI}
\left [ \frac{\partial E^*}{\partial v}(s,u,v) - 
\frac{\partial E^*}{\partial v}(s,\QI,v) \right ] \; -
\nonumber \\
& \frac{\varrho (1 - \QI)u}{(u-\QI)^2} 
\left[ E^*(s,u,v) - E^*(s,\QI,v) \right ].
\label{PDEF2}
\end{align}
While the second term of the right-hand side of (\ref{PDEF2}) remains well-defined at $u = \QI$, the third term has a polar singularity of order 1 at 
$u = \QI$. To circumvent the presence of singular terms in PDE (\ref{PDEF2}) for function $E^*$, we introduce the new function $F^*$ as defined in 
(\ref{defFTheta}). To express the derivatives 
$\partial E^*/\partial u$ and $\partial E^*/\partial v$ in terms of $F^*$, 
$\partial F^*/\partial u$ and $\partial F^*/\partial v$, successively differentiate definition relation (\ref{defFTheta}) with respect to $u$ and 
$v$ which readily provides
$$
\left\{
\begin{array}{ll}
\displaystyle \frac{\partial E^*}{\partial u}(s,u,v) = F^*(s,u,v) + (u-\QI) \frac{\partial F^*}{\partial u}(s,u,v), 
\\ \\
\displaystyle \frac{\partial E^*}{\partial v}(s,u,v) = 
(u-\QI) \frac{\partial F^*}{\partial v}(s,u,v) + 
\frac{\partial E^*}{\partial v}(s,\QI,v);
\end{array} \right.
$$
replacing the latter into (\ref{PDEF2}) and noting that the coefficient of $F^*(s,u,v)$ now equals
$$
u \cdot \frac{P(s,u)-\varrho(1-\QI)}{u-\QI} + (u+v)(u-\QI) = 
u(u-s-1-\varrho) + (u+v)(u-\QI)
$$
(with $P(s,\QI) = \varrho(1-\QI)$), the latter PDE reduces to (\ref{PDEF0}) after simple algebra.
\end{proof}

At this stage, we can successively note that

\begin{itemize}
\item[\textbf{a)}] equation (\ref{PDEF0}) for $F^*$ is of order 1 and linear 
(\cite{ARN15}, Lecture 1, Section 1.2), with smooth polynomial coefficients in both variables $u$ and $v$;

\item[\textbf{b)}] the last term $L(s,u,v)$ in (\ref{PDEF0}) involves the unknown function $E^*$ along with its derivative $\partial E^*/\partial v$ on the line $u = \QI$. 
\end{itemize}

\noindent
Considering this term $L(s,u,v)$ as known, equation (\ref{PDEF0}) can be integrated by using the method of characteristic curves applied in the next Section. Before addressing this integration, another simple variable change will  enable us to transform the quasi-linear equation (\ref{PDEF0}) into another simpler linear equation. 

\begin{corol}
\textbf{For given $s > 0$, let}
\begin{equation}
\Phi(s,u,v) = P(s,u) \, (1-v) \, F^*(s,u,uv), 
\qquad 0 < \vert u \vert < 1, \; \vert v \vert < 1,
\label{defPhi}
\end{equation}
\textbf{with polynomial $P$ introduced in (\ref{defPTheta}). Then function 
$\Phi$ satisfies the inhomogeneous linear PDE}
\begin{equation}
\frac{\partial \Phi}{\partial u} -  
\left [ \frac{(u - \QI)}{P(s,u)} \right] v(1-v) 
\frac{\partial \Phi}{\partial v} + \ell(s,u,v) = 0
\label{PDEF0bis}
\end{equation}
\textbf{where}
$$
\ell(s,u,v) = (1-v) \, \frac{L(s,u,uv)}{u},
$$
\textbf{with function $L$ defined in (\ref{defL})}.
\label{C1}
\end{corol}

\begin{proof}
For $(u,v) \in \mathbb{D}^2$ and $u \neq 0$, consider the variable change $(u,v) \mapsto (u,uv)$ and the auxiliary function 
$\Phi_0(s,\cdot,\cdot)$ defined by
\begin{equation}
\Phi_0(s,u,v) = F^*(s,u,uv), \qquad (u,v) \in \mathbb{D}^2, \; u \neq 0.
\label{DefPhi0}
\end{equation}
Applying the chain rule to (\ref{DefPhi0}), we readily calculate
\begin{align}
u \, \displaystyle \frac{\partial F^*}{\partial u}(s,u,uv) & \, = 
u \, \frac{\partial \Phi_0}{\partial u}(s,u,v) - 
v \, \frac{\partial \Phi_0}{\partial v}(s,u,v), 
\nonumber \\
u \, v \, \displaystyle \frac{\partial F^*}{\partial v}(s,u,uv) & \, = 
v \, \frac{\partial \Phi_0}{\partial v}(s,u,v).
\nonumber
\end{align}
From equation (\ref{PDEF0}) applied at point $(u,uv)$ and the latter identities, we then easily deduce that $\Phi_0(s,\cdot,\cdot)$ verifies the equation
\begin{align}
& \; u P(s,u) \, \frac{\partial \Phi_0}{\partial u}(s,u,v) -  
u(u - \QI)v(1-v) \, \frac{\partial \Phi_0}{\partial v}(s,u,v) \; +
\nonumber \\ 
& \left [ u(u-s-1-\varrho) + (u-\QI)(u + uv) \right ] \cdot \Phi_0(s,u,v) 
+ L(s,u,uv) = 0,
\label{PDEF0var1}
\end{align}
after using the definition (\ref{defPTheta}) of $P(s,u)$ to reduce the coefficient of $\partial \Phi_0/\partial v$ to 
$$
-v \, P(s,u) + v \, \left[ (u-\QI)(uv - s - 1 - \varrho) + \varrho(1-\QI) \right] = -u(u-\QI)v(1-v).
$$
Furthermore, writing the coefficient of $\Phi_0(s,u,v)$ as 
$u(P'(s,u) + (u-\QI)v)$ (where $P'(s,u)$ denotes for short the first derivative of $P(s,u)$ with respect to variable $u$) and dividing each side of (\ref{PDEF0var1}) by $u \neq 0$, the latter reduces to
\begin{equation}
P(s,u) \frac{\partial \Phi_0}{\partial u} - (u - \QI)v(1-v) \, 
\frac{\partial \Phi_0}{\partial v} + \left [ P'(s,u) + (u-\QI)v \right ] \Phi_0
+ \frac{L(s,u,uv)}{u} = 0
\label{PDEF0var2}
\end{equation}
(where $\Phi_0$ and all its derivatives are taken at point $(s,u,v)$). 

To eliminate the linear term in $\Phi_0$ in equation (\ref{PDEF0var2}), consider the function change $\Phi \mapsto \Phi_0$ where 
$\Phi_0 = M \times \Phi$ for some regular functions $M$. Following 
(\ref{PDEF0var2}), $\Phi$ should satisfy the equation
\begin{equation}
P(s,u) \frac{\partial \Phi}{\partial u} - (u - \QI)v(1-v) \, 
\frac{\partial \Phi}{\partial v} + A(s,u,v) \frac{\Phi}{M} + 
\frac{L(s,u,uv)}{u \cdot M} = 0
\label{PDEF0var3}
\end{equation}
where
$$
A(s,u,v) = P(s,u) \frac{\partial M}{\partial u} - (u - \QI)v(1-v) \, 
\frac{\partial M}{\partial v} + \left [ P'(s,u) + (u-\QI)v) \right ] M;
$$
the coefficient of $\Phi$ in (\ref{PDEF0var3}) therefore vanishes for any regular function $M$ verifying the homogeneous linear PDE defined by 
$A(s,u,v) = 0$; by easy inspection, a particular solution $M$ to that PDE can be chosen as
\begin{equation}
M(s,u,v) = \frac{1}{P(s,u)(1-v)}, \qquad 
0 < \vert u \vert < 1, \; \vert v \vert < 1y.
\label{Defm}
\end{equation}
From (\ref{DefPhi0}) and the determination (\ref{Defm}) of $M$, the corresponding function $\Phi$ is thus given by
$\Phi(s,u,v) = \Phi_0(s,u,v)/M(s,u,v) = P(s,u) (1-v) F^*(s,u,uv)$ 
for $0 < \vert u \vert < 1$, $\vert v \vert < 1$, as introduced in 
(\ref{defPhi}); dividing each side of (\ref{PDEF0var3}) by $P(s,u)$, this inhomogeneous linear equation for $\Phi$ reduces to equation (\ref{PDEF0bis}), as claimed. 
\end{proof}


\section{The solution $F^*$ along characteristic curves}


Let $a$, $b$, $c$ denote given continuous functions in some domain of  
$\mathbb{C}^2$ and consider the inhomogeneous linear PDE
\begin{equation}
a(u,v) \frac{\partial Z}{\partial u} + 
b(u,v) \frac{\partial Z}{\partial v} = c(u,v)
\label{GenPDE}
\end{equation}
with solution $Z:(u,v) \mapsto Z(u,v)$. Following (\cite{ARN15}, Lecture 1, Sections 1.2), basic properties of the solutions to equation 
(\ref{GenPDE}) can be recalled as follows:

\begin{itemize}

\item given a tuple $(u_0,v_0,z_0)$ with 
$\vert a(u_0,v_0) \vert ^2 + \vert b(u_0,v_0) \vert ^2 \neq 0$, 
the characteristic curve $\pmb{\gamma}_{u_0,v_0,z_0}$ of (\ref{GenPDE}) passing through the point $(u_0,v_0,z_0)$ is the solution 
$\tau \in \mathbb{R}^+ \mapsto (u(\tau),v(\tau),z(\tau))$ to the differential system 
\begin{equation}
\frac{\mathrm{d}u}{a(u,v)} = \frac{\mathrm{d}v}{b(u,v)} = 
\frac{\mathrm{d}z}{c(u,v)} = \mathrm{d}\tau
\label{DefCHAR}
\end{equation}
with initial condition $u(0) = u_0$, $v(0) = v_0$, $z(0) = z_0$. A first integral of system (\ref{DefCHAR}) is a real function 
$\mathbf{k}:(u,v,z) \mapsto \mathbf{k}(u,v,z)$ constant along any characteristic curve $\gamma_{u_0,v_0,z_0}$, that is, 
$\mathbf{k}(u(\tau),v(\tau),z(\tau)) = \mathbf{k}(u_0,v_0,z_0)$ for all 
$\tau \in \mathbb{R}^+$; 

\item let $\mathbf{k}_1$ and $\mathbf{k}_2$ be two independent first integrals of system (\ref{DefCHAR}). The characteristic $\pmb{\gamma}_{u_0,v_0,z_0}$ is then determined by the intersection of surfaces with respective equation
$\mathbf{k}_1(u,v,z) = \mathbf{k}_1(u_0,v_0,z_0)$ and
$\mathbf{k}_2(u,v,z) = \mathbf{k}_2(u_0,v_0,z_0)$;
besides, the general solution $Z:(u,v) \mapsto Z(u,v)$ to (\ref{GenPDE}) is implicitly defined by the relation 
\begin{equation}
\mathbf{k}_2(u,v,Z) = h(\mathbf{k}_1(u,v,Z))
\label{IntGEN}
\end{equation}
for $Z = Z(u,v)$, where $h:\mathbb{C} \rightarrow \mathbb{C}$ is any regular function.
\end{itemize}

\noindent
In this section, the characteristic curves associated with PDE 
(\ref{PDEF0bis}) are determined and the analytic solution $F^*$ to our initial PDE (\ref{PDEF0}) is derived accordingly. 

\subsection{Characteristic curves}
\label{CHAR2}

Fix $s > 0$ and introduce the coefficients
\begin{equation}
C^+(s) = - \frac{U^-(s) - \QI}{U^+(s)-U^-(s)},
\quad
C^-(s) = - \frac{U^+(s) - \QI}{U^-(s)-U^+(s)}
\label{C+-}
\end{equation}
where $U^+(s)$ and $U^-(s)$ are the roots of quadratic polynomial  
$P(s,\cdot)$ given in (\ref{defPTheta}); as already shown in 
(\cite{GuiSim18}, Section 4.1, Equ.(4.15)), coefficients $C^\pm(s)$ verify
\begin{equation}
\forall \; s > 0, \quad C^+(s) < 0 < 1 < C^-(s).
\label{InequC+-}
\end{equation}
The following lemma first states the analyticity of a related function in the disk $\mathbb{D}$ cut along a linear segment (see the Proof in Appendix 
\ref{App1}).

\begin{lemma} 
\textbf{Given $u_0 \in \mathbb{D} \setminus \{U^-(s)\}$, let 
$\Lambda_{u_0}$ denote the line segment starting at point $U^-(s)$ and directed along the vector $(u_0,U^-(s))$. 
The function $\mathfrak{R}(s,u_0;\cdot)$ defined by}
\begin{equation}
\mathfrak{R}(s,u_0;u) = 
\left ( \frac{u-U^-(s)}{u_0-U^-(s)} \right )^{C^-(s)-1}
\left ( \frac{u-U^+(s)}{u_0-U^+(s)} \right )^{C^+(s)-1}, 
\; u \in \mathbb{D} \setminus \Lambda_{u_0},
\label{defR}
\end{equation}
\textbf{is analytic on the cut disk $\mathbb{D} \setminus \Lambda_{u_0}$}. 
\label{LemJ}
\end{lemma}

Let us now determine the characteristics for the simpler linear PDE 
(\ref{PDEF0bis}); its associated differential system (\ref{DefCHAR}) for characteristic curves $\pmb{\gamma}_{u_0,v_0,z_0}$ in the 
$(O,Ou,Ov,Oz)$ space reads
\begin{equation}
\frac{\mathrm{d}u}{1} = 
- \, \frac{P(s,u)}{(u-\QI)v(1-v)} \, \mathrm{d}v = 
- \, \frac{u}{(1-v) \, L(s,u,uv)} \, \mathrm{d}z.
\label{CAR1}
\end{equation}
This system can be solved as follows. 

\begin{lemma}
\textbf{A) In the $(O,Ou,Ov,Oz)$ space, the characteristic curve 
$\pmb{\gamma}_{u_0,v_0,z_0}$ of PDE (\ref{PDEF0bis}) is the intersection of surfaces with equation $\mathbf{k}_1(u,v) = C^{st}$ and 
$\mathbf{k}_2(u,v,z) = C^{st}$, respectively, where $\mathbf{k}_1$ and $\mathbf{k}_2$ are the independent first integrals to system (\ref{CAR1}) defined by}
\begin{equation}
\left\{
\begin{array}{ll}
\mathbf{k}_1(u,v) = \displaystyle \frac{v}
{(u-U^-(s))^{C^-(s)-1}(u-U^+(s))^{C^+(s)-1}}, 
\\ \\
\mathbf{k}_2(u,v,z) = \displaystyle 
z - \int_u^{U^-(s)} L(s,\zeta,\zeta \, v \, \mathfrak{R}(s,u;\zeta)) \, 
e^{-v \, \mathfrak{R}(s,u;\zeta)} \, \frac{\mathrm{d}\zeta}{\zeta}
\end{array} \right.
\label{DefI0I1}
\end{equation}
\textbf{for $u \in \mathbb{D} \setminus [U^-(s),1]$, $v \in \mathbb{C}$ and with $\mathfrak{R}(s,u;\zeta)$, $\zeta \in \mathbb{D} \setminus \Lambda_u$, introduced in (\ref{defR})}.

\textbf{B) The projection of characteristic $\pmb{\gamma}_{u_0,v_0,z_0}$ on the 
$(O,Ou,Ov)$-plane has the Cartesian equation} 
\begin{equation}
v = v_0 \cdot \mathfrak{R}(s,u_0;u), \qquad 
u \in \mathbb{D} \setminus \Lambda_{u_0},
\label{EquCAR1}
\end{equation}
\textbf{and always passes through the fixed point $(U^-(s),0)$}.
\label{LemCAR}
\end{lemma}
\noindent
We refer to Appendix \ref{App2} for the detailed proof of 
Lemma \ref{LemCAR} (the dependence of the first integrals $\mathbf{k}_1$ and $\mathbf{k}_2$ on parameter $s$ is not mentioned here for conciseness of notation).

\subsection{Integral representation of the analytic solution}
Using Lemma \ref{LemCAR} and assuming that the function $L$ is known, we can now derive an integral representation of the solution $F^*$ to PDE 
(\ref{PDEF0}) which is analytic in some relevant domain.

\begin{prop}
\textbf{Given the function $L$ defined in (\ref{defL}), the solution 
$F^*(s,\cdot,\cdot)$ to PDE (\ref{PDEF0}) which is analytic in the product 
$]0,U^-(s)[ \times \mathbb{C}$ and vanishes on the line $v = 0$ can be expressed by}
\begin{equation}
F^*(s,u,v) = \frac{e^{\frac{v}{u}}}{P(s,u)} 
\int_u^{U^-(s)} L\left(s,\zeta,\frac{\zeta}{u} \, v \, \mathfrak{R}(s,u;\zeta)\right) 
\, e^{- \frac{v}{u} \, \mathfrak{R}(s,u;\zeta)} \, 
\dfrac{\mathrm{d}\zeta}{\zeta} 
\label{IntGen0}
\end{equation}
\textbf{for $s > 0$ and $u \in \; ]0,U^-(s)[$, $v \in \mathbb{C}$}.
\label{Prop0bis}
\end{prop}

\begin{proof}
For any locally regular function $h:\mathbb{C} \rightarrow \mathbb{C}$, the relation (\ref{IntGEN}) between the two first integrals $\mathbf{k}_1$ and 
$\mathbf{k}_2$ determines a solution $\Phi(s,\cdot,\cdot)$ to PDE 
(\ref{PDEF0bis}). Given the specific expressions (\ref{DefI0I1}) of 
$\mathbf{k}_1$ and $\mathbf{k}_2$, the general expression of $z = \Phi(s,u,v)$ is consequently given by
\begin{equation}
\Phi(s,u,v) = h(\mathbf{k}_1(u,v)) + \mathbf{G}(s,u,v) 
\label{IntGen1}
\end{equation}
where $\mathbf{G}(s,u,v)$ denotes the integral term
\begin{equation}
\mathbf{G}(s,u,v) = 
\int_u^{U^-(s)} \frac{L(s,\zeta,\zeta \, v \, \mathfrak{R}(s,u;\zeta))}{\zeta} \, 
e^{-v \, \mathfrak{R}(s,u;\zeta)} \mathrm{d}\zeta.
\label{IntG}
\end{equation}
To specify an analyticity domain for function $\mathbf{G}(s,\cdot,\cdot)$, consider the real interval $J = \; ]0,U^-(s)[$; the function 
$(u,v) \in J \times \mathbb{D} \mapsto \mathbf{G}(s,u,v)$ is then analytic (in fact, Lemma \ref{LemJ} ensures that each function $\mathfrak{R}(s,u;\cdot)$, 
$u \in J$, is analytic in $J$ and definition (\ref{defL}) entails that 
$L(s,\cdot,\cdot)$ is analytic in $\mathbb{D} \times \mathbb{C}$ so that the integrand in (\ref{IntG}) is an analytic function of the pair 
$(u,\zeta) \in J \times J$). 

On the other hand, use the variable change
\begin{equation}
\zeta = \zeta_u(t) = u + t(U^-(s) - u), \quad 0 \leq t \leq 1,
\label{defZeta}
\end{equation}
in integral (\ref{IntG}) to obtain
\begin{equation}
\mathbf{G}(s,u,v) = (U^-(s) - u) 
\int_0^{1} \frac{L(s,\zeta_u(t),\zeta_u(t) \, v \, 
\mathbf{R}(s,u;t))}{\zeta_u(t)} \, 
e^{-v \, \mathbf{R}(s,u;t)} \, \mathrm{d}t
\label{IntG1}
\end{equation}
where, after definition (\ref{defR}),
\begin{equation}
\mathbf{R}(s,u;t) = \mathfrak{R}(s,u;\zeta_u(t)) = 
(1-t)^{C^-(s)-1}
\left ( 1 - t \cdot \frac{U^-(s)-u}{U^+(s) - u }\right )^{C^+(s)-1}.
\label{defRsut}
\end{equation}
The equivalent expression (\ref{IntG1}) of $\mathbf{G}(s,u,v)$ and the inequality $C^-(s) > 1$ after (\ref{InequC+-}) together imply that, for any given $v \in \mathbb{D}$,
\begin{equation}
\mathbf{G}(s,u,v) = O(U^-(s) - u) \qquad 
\mathrm{when} \qquad u \uparrow U^-(s).
\label{EquG}
\end{equation}

Let us now determine the first term $h(\mathbf{k}_1(u,v))$ in (\ref{IntGen1}). By definition of the first integral $\mathbf{k}_1$, this term is constant along the characteristic curve passing through the point $(u,v)$; also recall from Lemma \ref{LemCAR}.\textbf{B} that this characteristic always passes through the fixed point $(U^-(s),0)$ so that, in particular, 
\begin{equation}
h(\mathbf{k}_1(u,v)) = h(\mathbf{k}_1(U^-(s),0)).
\label{hk1}
\end{equation}
Besides, the definition relation (\ref{defPhi}) between functions $\Phi$ and 
$F^*$ together with the boundary condition (\ref{defFbc}) for $F^*$ imply that 
$\Phi$ also vanishes on the line $v = 0$. Applying relation (\ref{IntGen1}) at 
point $(u,v) = (U^-(s),0)$, the latter discussion entails 
$0 = \Phi(s,U^-(s),0) = h(\mathbf{k}_1(U^-(s),0)) + \mathbf{G}(s,U^-(s),0)$ 
which, after (\ref{EquG}) and (\ref{hk1}), yields $h(\mathbf{k}_1(u,v))  = 0$. Finally, by the variable change $(u,v) \mapsto (u,uv)$ of definition 
(\ref{defPhi}), (\ref{IntGen1}) provides
\begin{equation}
F^*(s,u,v) = \Phi \left ( s, u, \frac{v}{u}\right ) 
= \frac{e^{\frac{v}{u}}}{P(s,u)} \cdot 
\mathbf{G}\left(s,u,\frac{v}{u}\right) 
\label{IntGen2}
\end{equation}
and expression (\ref{IntGen0}) follows. 
\end{proof}


\section{Determination of function $L$}


To proceed with the resolution to PDE (\ref{PDEF0}), we are left to determine the function $L$ involved in integral representation (\ref{IntGen0}) or, equivalently, the function $E^*(s,\QI,\cdot)$ for given $s > 0$.

\subsection{Integral condition on $L$}
At this stage, representation (\ref{IntGen0}) of $F^*(s,u,v)$ is restricted to $u \in \, ]0,U^-(s)[$, while it is known that $F^*(s,u,v)$ is obviously well-defined near $u = 0$ and $u = U^-(s)$ for any $v \in \mathbb{C}$. We now establish a necessary and sufficient condition on $L$ to ensure that expression (\ref{IntGen0}) of $F^*$ to be defined and analytic at point 
$u = 0$. 

\begin{prop}
\textbf{Condition}
\begin{equation}
\forall \; v \in \mathbb{C}, \qquad 
\int_0^{U^-(s)} L\left(s,\zeta,\zeta \, v \, \mathfrak{R}(s,0;\zeta)\right) 
\, e^{- v \, \mathfrak{R}(s,0;\zeta)} \, 
\dfrac{\mathrm{d}\zeta}{\zeta} = 0
\label{CNC0}
\end{equation}
\textbf{on function $L$ defined in (\ref{defL}) is necessary and sufficient for ensuring the analyticity of $F^*$ at point $u = 0$}.
\label{CNCL}
\end{prop}

\noindent
We refer to Appendix \ref{App3} for the proof of Proposition \ref{CNCL}. 
Note that, using the variable change $(u,v) \mapsto (u,uv)$ in 
(\ref{IntGen0}) readily shows that
\begin{equation}
F^*(s,u,uv) = \frac{e^v}{P(s,u)} 
\int_u^{U^-(s)} L\left(s,\zeta,\zeta \, v \, \mathfrak{R}(s,u;\zeta)\right) 
\, e^{- v \, \mathfrak{R}(s,u;\zeta)} \, 
\dfrac{\mathrm{d}\zeta}{\zeta}
\label{CNC0bis}
\end{equation}
is well-defined for $u = 0$ and boundary condition (\ref{defFbc}) entails that its value is zero; as $P(s,0) = s\QI + \varrho + \QI \neq 0$ for $s > 0$, expression (\ref{CNC0bis}) thus also provides equation (\ref{CNC0}) as a necessary condition for the function $L$ to exist.

\subsection{Determination of function $E^*(s,\QI,\cdot)$}
Proposition \ref{CNCL} now translates into the following assertion showing that equation (\ref{CNC0}) has a unique entire solution $E^*(s,\QI,\cdot)$ on the whole plane $\mathbb{C}$.

\begin{prop}
\textbf{The coefficients $E_b^*(s,\QI)$, $b \geq 1$, of the power series 
$E^*(s,\QI,v) = \sum_{b \geq 1} E_b^*(s,\QI) v^b/b!$, $v \in \mathbb{C}$, are determined by the triangular linear system}
\begin{equation}
\forall \, b \geq 1, \qquad 
\sum_{\ell = 1}^b (-1)^\ell \binom{b}{\ell} \, 
Q_{b,\ell}(s) \, E_{\ell}^*(s,\QI) = K_b(s)
\label{T0}
\end{equation}
\textbf{with coefficients} $Q_{b,\ell}(s) = (\ell + 1 - b) M_{b,\ell}(s) - 
\ell (s + 1 + \varrho) M_{b,\ell-1}(s)$ \textbf{where}
$$
M_{b,\ell}(s) = \displaystyle 
\int_0^{U^-(s)} \zeta^\ell \, 
\mathfrak{R}(s,0;\zeta)^b \, \mathrm{d}\zeta,
$$
\textbf{and with right-hand side}
$$
K_b(s) = 
\int_0^{U^-(s)} \left[(b-1)(1-\zeta)^b + 1\right] \, 
\mathfrak{R}(s,0;\zeta)^b \, \frac{\mathrm{d}\zeta}{(1-\zeta)^2}.
$$
\end{prop}

\begin{proof}
From the definition (\ref{defL}) of $L$ and after expanding $E^*(s,\QI,v)$ into its power series of variable $v$, we can write
$$
\frac{L(u,uv)}{v} = \frac{e^{uv}}{(1-u)^2} 
\left [ 1 - e^{-uv} + v(1-u) \right ] + 
\sum_{b \geq 1} \Lambda_b(s,u) \frac{{v}^b}{b!}
$$
with $\Lambda_b(s,u) = u^b E_b(s,\QI) + b \, u^{b-1} E_{b-1}^*(s,\QI) - 
b(s+1+\varrho)u^{b-1}E_b^*(s,\QI)$ for short; using this expression for the integrand in the left-hand side of condition (\ref{CNC0}), the latter reads in the form $A(s,v) + B(s,v) = 0$ with
$$
A(s,v) = \int_0^{U^-(s)} \left [ \left\{ 1+\mathfrak{R}(s,0;\zeta) \, v \, 
(1-\zeta) \right\} e^{-v(1-\zeta) \, \mathfrak{R}(s,0;\zeta)} - 
{e^{-v \, \mathfrak{R}(s,0;\zeta)}} \right ] 
\frac{\mathrm{d}\zeta}{(1-\zeta)^2}
$$
and
$$
B(s,v) = \int_0^{U^-(s)} \left[ \sum_{b \geq 1} \Lambda_b(s,\zeta)  
\, \frac{\mathfrak{R}(s,0;\zeta)^b \, v^b}{b!} \right ]
\, e^{-\mathfrak{R}(s,0;\zeta) \, v} \, \mathrm{d}\zeta.
$$
Expanding the exponential in the integrand of $A(s,v)$ into a power series of variable $v$ first easily gives
\begin{equation}
A(s,v) = - \sum_{b \geq 1} (-1)^b K_b(s) \, \frac{v^b}{b!}
\label{Asv}
\end{equation}
with coefficient $K_b(s)$ given as in the Proposition; besides, expanding the exponential factor $e^{-\mathfrak{R}(s,0;\zeta) \, v}$ of the integrand in 
$B(s,v)$ above into a power series of variable $v$ readily gives the expansion
\begin{equation}
B(s,v) = \sum_{b \geq 0} (-1)^b \int_0^{U^-(s)} 
\left [ \sum_{\ell=0}^b (-1)^{\ell} \Lambda_\ell(s,\zeta) \binom{b}{\ell} \right ] 
\frac{\mathfrak{R}(s,0;\zeta)^b v^b}{b!} \, \mathrm{d}\zeta
\label{Bsv}
\end{equation}
(noting that $\Lambda_0(s,\zeta) = 0$ since $E_0^*(s,\QI) = 0$). On account of expansions (\ref{Asv}) and (\ref{Bsv}) (together with the latter definition of the integrand $\Lambda_d(s,\zeta)$), relation $A(s,v) + B(s,v) = 0$ is therefore equivalent to
$$
\sum_{\ell=1}^b (-1)^\ell \binom{b}{\ell} B_{b,\ell}(s) E_\ell^*(s,\QI) + 
\sum_{\ell=1}^b (-1)^\ell \binom{b}{\ell} \ell \, M_{b,\ell-1}(s) 
E_{\ell-1}^*(s,\QI) = K_b(s)
$$
for all $b \geq 1$ and with  
$B_{b,\ell}(s) = M_{b,\ell}(s) - \ell(s+1+\varrho)M_{b,\ell-1}(s)$, 
where $M_{b,\ell}(s)$ is the definite integrals defined as in the Proposition. By simply changing the index in the second sum above and noting that 
$\binom{b}{\ell+1} = (b-\ell)\binom{b}{\ell}/(\ell+1)$, the latter equation reduces to (\ref{T0}). This defines a triangular linear system for all coefficients $E_b^*(s,\QI)$, $b \geq 1$, as claimed.
\end{proof}

Each coefficient $M_{b,\ell}(s)$, $1 \leq \ell \leq b$, can be expressed in terms of the Gauss hypergeometric function. Specifically, recall the integral representation (\cite{NIST10}, Chap.15, 15.6.1)
\begin{equation}
F(\alpha,\beta;\gamma;z) = 
\frac{\Gamma(\gamma)}{\Gamma(\beta)\Gamma(\gamma-\beta)}
\int_0^1 \frac{t^{\beta-1}(1-t)^{\gamma-\beta-1}}{(1-z t)^\alpha} 
\, \mathrm{d}t, \quad z \in \mathbb{D},
\label{HyperGauss}
\end{equation}
of the Gauss hypergeometric function $F(\alpha,\beta;\gamma;\cdot)$ 
with parameters $\alpha$, $\beta$ and $\gamma > \beta$. 




\section{Appendix}


\subsection{Proof of Proposition \ref{PDE1_Batch}}
\label{App0}
\textbf{A)} As a preliminary result, we first state the following lemma for a partial differential system verified by the collection of generating functions $E_b$, 
$b \geq 1$, introduced in (\ref{defEb}). 

\begin{lemma}
\textbf{If the distribution of the batch size is geometric with parameter 
$\QI \in [0,1[$, the generating functions $E_b$, $b \geq 1$, verify the partial differential system}
\begin{align}
u \frac{\partial^2 E_b}{\partial u \partial x}(x,u) + b \frac{\partial E_b}{\partial x}(x,u) = & \; 
\frac{\varrho(1-\QI)(b\QI-(b-1)u)}{(u-\QI)^2}
\left[ E_b(x,\QI) - E_b(x,u) \right ] \; +
\nonumber \\
& \; \frac{u(u-1)(u-\varrho-\QI)}{u-\QI}\frac{\partial E_b}{\partial u}(x,u) + 
b E_{b-1}(x,u) \cdot \mathbf{1}_{b \geq 2} \; +
\nonumber \\
& \; \left [ \, u - (1+\varrho)b \, \right ]E_b(x,u)
\label{PartEb}
\end{align}
\textbf{for all $b \geq 1$ and $x \in \mathbb{R}^+$, $u \in \mathbb{D}$}.
\label{PartialEb}
\end{lemma}

\begin{proof}
For given $b \geq 1$, multiply each side of equation (\ref{EQU0_Batch}) by 
$(n+b)u^n$, $n \geq 0$, and then sum all these equations over $n \geq 0$. Assuming that distribution $(q_b)_{b \geq 1}$ is geometric with parameter 
$\QI$ as in (\ref{defGeo}) and applying the definition (\ref{defEb}) for the generating function $E_b$, $b \geq1$, we obtain
\begin{align}
& u \frac{\partial^2 E_b}{\partial u \partial x}(x,u) + 
b \frac{\partial E_b}{\partial x}(x,u) =
 \varrho (1-\QI) R_b(x,u) + \varrho b (1-\QI) S_b(x,u) \; -
\nonumber \\
& (1 + \varrho) u \frac{\partial E_b}{\partial u}(x,u) - 
(1 + \varrho) b E_b(x,u) + T_b(x,u) + b E_{b-1}(x,u)
\label{PDEb0}
\end{align}
where
$$
\left\{
\begin{array}{ll}
R_b(x,u) = \displaystyle 
\sum_{n \geq 0}\sum_{m \geq 1}\QI^{m-1} n \, u^n E_{n+m,b}(x), \quad 
T_b(x,u) = \displaystyle \sum_{n \geq 0} n \, u^n E_{n-1,b}(x),
\\ \\
S_b(x,u) = \displaystyle 
\sum_{n \geq 0}\sum_{m \geq 1}\QI^{m-1} u^n E_{n+m,b}(x)
\end{array} \right.
$$
for all $x \in \mathbb{R}^+$ and $u \in \mathbb{D}$. The series $R_b(x,u)$, 
$S_b(x,u)$ and $T_b(x,u)$ intervening in the right-hand side of equality 
(\ref{PDEb0}) can be successively calculated as
\begin{align}
R_b(x,u) = & \; 
\sum_{k \geq 1}\sum_{n=0}^{k-1} \QI^{k-n-1} n \, u^n E_{k,b}(x)
= \sum_{k \geq 1} \QI^{k-1} \frac{u}{\QI} \cdot 
h_k \left ( \frac{u}{\QI} \right ) \cdot E_{k,b}(x)
\nonumber
\end{align} 
where we set 
$h_k(r) = \mathrm{d} ( \sum_{n=0}^{k-1} r^n ) / \mathrm{d}r = 
(1-r^k)/(1-r)^2 -  k \, r^{k-1}/(1-r)$ with $r = u/\QI$, so that we eventually obtain
$$
R_b(x,u) = - u \, \frac{E_b(x,u) - E_b(x,\QI)}{(u-\QI)^2} + \frac{u}{u-\QI} \, \frac{\partial E_b}{\partial u}(x,u); 
$$
besides,
\begin{align}
S_b(x,u) = & \; 
\sum_{k \geq 1}\sum_{n=0}^{k-1} \QI^{k-n-1} u^n E_{k,b}(x)
= \sum_{k \geq 1} \QI^{k-1} \frac{1 - \displaystyle \left(u/\QI\right)^k}
{1- \displaystyle u/\QI}E_{k,b}(x)
\nonumber \\
= & \; \frac{E_b(x,u)-E_b(\QI,u)}{u-\QI}
\nonumber
\end{align}
and finally
$$
T_b(x,u) = \sum_{k \geq 1} (k+1) \, u^{k+1} E_{k,b}(x) = 
u^2 \frac{\partial E_b}{\partial u}(x,u) + u E_b(x,u).
$$
Replacing these values of $R_b(x,u)$, $S_b(x,u)$ and $T_b(x,u)$ in the right-hand side of (\ref{PDEb0}) and factorizing the coefficient of the first derivative $\partial E_b(x,u)/\partial u$ as
$$
\frac{u}{u-\QI}(u^2 - (u+1+\varrho)u + \varrho + \QI) = \frac{u(u-1)(u-\varrho-\QI)}{u-\QI},
$$
the latter readily reduces to partial differential equation (\ref{PartEb}).
\end{proof}

\textbf{B)} We now turn to the proof of Proposition \ref{PDE1_Batch}. Multiply each side of equation (\ref{PartEb}) by $v^b$, $b \geq 1$, and then sum all these equations side by side over $b \geq 1$. On account of identities
$$
\sum_{b \geq 1} b \, E_b v^b = v \, \frac{\partial E}{\partial v}, 
\quad  
\sum_{b \geq 1} b \, E_{b-1} v^b = v^2 \frac{\partial E}{\partial v} + v \, E, \quad 
\sum_{b \geq 1} (b-1) \, E_b v^b = v \, \frac{\partial E}{\partial v} - E
$$
at any point $(x,u,v)$, $x \geq 0$, $(u,v) \in \mathbb{D} \times \mathbb{C}$, we then obtain
\begin{align}
& u \frac{\partial^2 E}{\partial u \partial x}(x,u,v) + 
v \frac{\partial^2 E}{\partial v \partial x}(x,u,v) = 
- \frac{\varrho(1-\QI)}{(u-\QI)^2} \biggl[ \QI 
\left ( v \frac{\partial E}{\partial v}(x,u,v) - 
v \frac{\partial E}{\partial v}(x,\QI,v) \right ) \; - 
\nonumber \\
& u \left ( \left \{ 
v \frac{\partial E}{\partial v}(x,u,v) - E(x,u,v) \right \} - 
\left \{ v \frac{\partial E}{\partial v}(x,\QI,v) - E(x,\QI,v) \right \} 
\right ) \biggr] \; +
\nonumber \\
& \frac{u(u-1)(u - \varrho - \QI)}{u - \QI}\frac{\partial E}{\partial u}(x,u,v) + u E(x,u,v) - (1+\varrho)v \frac{\partial E}{\partial v}(x,u,v) \; +
\nonumber \\
& v^2 \frac{\partial E}{\partial v}(x,u,v) + v E(x,u,v).
\nonumber
\end{align}
Reassembling all the factors of derivative $\partial E(x,u,v)/\partial v$ at point $(x,u,v)$ (resp. factors of derivative
$\partial E(x,\QI,v)/\partial v$ at point $(x,\QI,v)$) inside the bracket 
$$
\left [ \QI \left (v \, \frac{\partial E}{\partial v}(x,u,v) - ... \right ) 
- u \left ( \left \{ v \, \frac{\partial E}{\partial v}(x,u,v) - ... \right \} \right ) \right ]
$$
in the left-hand side of the latter equality, these factors eventually gather as  
$$
v(\QI-u)\frac{\partial E}{\partial v}(x,u,v) \qquad \mathrm{(resp. \; as} 
\quad 
v(\QI-u)\frac{\partial E}{\partial v}(x,\QI,v)).
$$
This easily leads to equation (\ref{EQU1_BatchE}), as claimed $\blacksquare$

\subsection{Proof of Lemma \ref{LemJ}}
\label{App1}
Fix $u_0 \in \mathbb{D} \setminus \{U^-(s)\}$. Using the arguments invoked in (\cite{GuiSim18}, Section 9.4), we can assert that the ratio
$(u-U^-(s))/(u_0-U^-(s))$ is non real negative at any point 
$u \in \mathbb{D} \setminus \Lambda_{u_0}$, where $\Lambda_{u_0}$ denotes the line segment starting at point $U^-(s) \in \mathbb{D}$ and directed along the vector $(u_0,U^-(s))$. 


Besides, the ratio $(u-U^+(s))/(u_0-U^+(s))$ is non real negative at any point $u \in \mathbb{D} \setminus \Lambda'_{u_0}$, where $\Lambda'_{u_0}$ denotes the line segment starting at point $U^+(s)$ and directed along the vector $(u_0,U^+(s))$; however, $U^+(s) \notin \mathbb{D}$ after inequalities (\ref{INEQ-U}), therefore 
$\mathbb{D} \cap \Lambda'_{u_0} = \emptyset$ and 
the latter ratio is consequently non negative for any $u \in \mathbb{D}$. We thus conclude that, for any given $u_0 \in \mathbb{D}$, the function 
$\mathfrak{R}(s,u_0;\cdot)$ introduced in (\ref{defR}) is well-defined and analytic in the cut disk $\mathbb{D} \setminus \Lambda_{u_0}$ 
$\blacksquare$

\subsection{Proof of Lemma \ref{LemCAR}}
\label{App2}
\textbf{A)} Let us determine two independent first integrals of the 
3-dimensional differential system (\ref{CAR1}).

\textbf{A.1} The first equation in (\ref{CAR1}) equivalently reads
\begin{equation}
\frac{\QI-u}{P(s,u)} \, \mathrm{d}u = \frac{\mathrm{d}v}{v}.
\label{CAR1a}
\end{equation}
Following the definition (\ref{C+-}) of coefficients $C^\pm(s)$, it readily follows that the rational fraction $(\QI-u)/P(s,u)$ with simple poles at 
$u = U^-(s)$ and $u = U^+(s)$, can be decomposed as
$$
\frac{\QI-u}{P(s,u)} = 
- \, \frac{C^+(s)}{u-U^-(s)} - \, \frac{C^-(s)}{u-U^+(s)} 
= \frac{C^-(s)-1}{u-U^-(s)} +\, \frac{C^+(s)-1}{u-U^+(s)}
$$
after using the identity $C^+(s) + C^-(s) = 1$. Differential equation 
(\ref{CAR1a}) then easily integrates to 
$v = \mathbf{k}_1(u_0,v_0) (u-U^-(s))^{C^-(s)-1}(u-U^+(s))^{C^+(s)-1}$ 
with integration constant $\mathbf{k}_1(u_0,v_0)$ given as in (\ref{DefI0I1}), thus defining the first integral $\mathbf{k}_1$. Given any point 
$(u_0,v_0,z_0)$, the projection on the $(O,Ou,Ov)$-plane of the characteristic curve $\pmb{\gamma}_{u_0,v_0,z_0}$ is independent of $z_0$ and its Cartesian equation is provided by the first integral $\mathbf{k}_1$, that is, 
$\mathbf{k}_1(u,v) = \mathbf{k}_1(u_0,v_0)$ which reduces to (\ref{DefI0I1}), with the function $\mathfrak{R}(s,u_0;\cdot)$ introduced in (\ref{defR}).

\textbf{A.2} Now replacing variable $v$ by the expression (\ref{EquCAR1})  obtained above, the two extreme sides of equation (\ref{CAR1}) give in turn
$$
\frac{L(s,u,u \, v_0 \, \mathfrak{R}(s,u_0;u))}{u} \cdot 
e^{-v_0 \, \mathfrak{R}(s,u_0;u)} \mathrm{d}u = - \mathrm{d}z
$$
which readily integrates to
\begin{align}
\mathbf{k}_2(u_0,v_0,z_0) = & \; z - \int_u^{U^-(s)} \frac{L(s,\zeta,\zeta \, v_0 \, 
\mathfrak{R}_{u_0}(s,\zeta))}{\zeta} \cdot 
e^{-v_0 \, \mathfrak{R}(s,u_0;\zeta)} \mathrm{d}\zeta
\nonumber \\
= & \; z - \int_u^{U^-(s)} 
\frac{L(s,\zeta,\zeta \, v \, \mathfrak{R}(s,u;\zeta))}{\zeta} 
\cdot e^{-v \, \mathfrak{R}(s,u;\zeta)} \mathrm{d}\zeta
\label{DefW1}
\end{align}
after replacing $v_0$ by its expression $v_0 = v/\mathfrak{R}(s,u_0;u)$ provided by (\ref{EquCAR1}) and using the identity  
$\mathfrak{R}(s,u_0;\zeta)/\mathfrak{R}(s,u_0;u) = \mathfrak{R}(s,u;\zeta)$ 
after definition (\ref{defR}). The expression (\ref{DefI0I1}) for the first integral $\mathbf{k}_2$ follows.

\textbf{B)} As $C^-(s) > 1$ by inequalities (\ref{InequC+-}), we readily have  
$\mathfrak{R}(s,u_0;U^-(s)) = 0$ for any 
$u_0 \in \mathbb{D} \setminus \{U^-(s)\}$. Equation (\ref{EquCAR1}) thus entails that the projection of any characteristic curve 
$\pmb{\gamma}_{u_0,v_0,z_0}$ onto the $(O,Ou,Ov)$ plane passes through the fixed point $(U^-(s),0)$ $\blacksquare$

\subsection{Proof of Proposition \ref{CNCL}}
\label{App3}
\textbf{A)} To prove Proposition \ref{CNCL}, consider again the collection of functions $E_b$, $b \geq 1$, introduced in (\ref{defEb}) and define the one-sided Laplace transform $E_b^*$ of each $E_b$ by
\begin{equation}
E_b^*(s,u) = \int_0^{+\infty} E_b(x,u) e^{-sx} \mathrm{d}x, 
\qquad s > 0,
\label{defLaplEb}
\end{equation}
for given $u \in \mathbb{D}$. As in Section \ref{PDEL} for the function change $E^* \mapsto F^*$, we here introduce the function change $E_b^* \mapsto F_b^*$ for each given $b \geq 1$, where the analytic function $F_b^*$ is defined by
\begin{equation}
F_b^*(s,u) = 
\left\{
\begin{array}{ll}
\displaystyle \frac{E_b^*(s,u) - E_b^*(s,\QI)}{u-\QI}, \quad \quad 
s \geq 0, \; u \in \mathbb{D} \setminus \{\QI\},
\\ \\
\displaystyle \displaystyle \frac{\partial E_b^*}{\partial u}(s,\QI), \quad \quad \quad \quad \quad \quad
s \geq 0, \; u = \QI.
\end{array} \right.
\label{defFbTheta}
\end{equation}

\begin{lemma}
\textbf{For $b \geq 1$, function $F_b^*$ can be expressed in terms of 
$E_{b-1}^*$ by}
\begin{align}
F_b^*(s,u) = \frac{1}{u^b \, P(s,u)} \times 
\int_u^{U^-(s)} & \Bigl[ \frac{z + b(1-z)}{(1-z)^2} + b \, E_{b-1}^*(z) \; + 
\nonumber \\
& ([z - b(1+\varrho+s)]E_b^*(s,\QI) \Bigr] \mathfrak{R}(s,u;z)^{b}z^{b-1} \, \mathrm{d}z
\label{Fb*}
\end{align}
\textbf{for $u \in \mathbb{D}$ and where $\mathfrak{R}$ is defined in 
(\ref{defR})}.
\label{ExpressFb*}
\end{lemma}

\begin{proof}
Recall by Lemma \ref{PartialEb} that functions $E_b$ verify the partial differential system (\ref{PartEb}). To relate this system to function $F_b^*$, $b \geq 1$, take the Laplace transform of each side of equation (\ref{PartEb}) for given $u \in \mathbb{D}$; noting that $E_b(0,u) = 1/(1-u)$ after condition 
(\ref{PW1}) and following a pattern similar to the proof of Corollary 
\ref{Corollary1}, we derive the first order system  
\begin{align}
- & \frac{u \, P(s,u)}{u-q} \, \frac{\partial E_b^*}{\partial u}(s,u) + 
\left[b(1+\varrho+s) - u\right] E_b^*(s,u) \; - 
\nonumber \\
& \, \frac{\varrho(1-\QI)(b\QI - (b-1)u)}{(u-\QI)^2} 
\left[ E_b^*(s,\QI) - E_b^*(s,u) \right] \; =
\nonumber \\
& \, \frac{u + b(1-u)}{(1-u)^2} + b \, E_{b-1}^*(s,u) \mathbf{1}_{b \geq 2}, 
\qquad u \in \mathbb{D},
\label{EbL}
\end{align}
for all transforms $E_b^*$, $b \geq 1$. Expressing each equation 
(\ref{EbL}) in terms of $F_b^*$ after (\ref{defFbTheta}) then cancels out all denominators in $1/(u-\QI)$ or $1/(u-\QI)^2$ and we then eventually obtain
\begin{align}
& u \, P(s,u) \frac{\partial F_b^*}{\partial u}(s,u) + Q_b(s,u) F_b^*(s,u) \; = 
\nonumber \\
& - \frac{u + b(1-u)}{(1-u)^2} - b E_{b-1}^*(s,u) \mathbf{1}_{b \geq 2} - 
\left [ u - b(1+\varrho+s) \right ] E_b^*(s,\QI)
\label{FbL}
\end{align}
after some algebraic reduction and where $Q_b(s,u)$ denotes the quadratic polynomial 
$Q_b(s,u) = P(s,u) + u^2 - b(1+\varrho+s)u + (b-1)(s\QI + \varrho + \QI)$. For given $s > 0$, we now solve the first order differential equation (\ref{FbL}) for $F_b^*$; noting that 
$$
- \frac{Q_b(s,u)}{u \, P(s,u)} = 
- \frac{b}{u} + \frac{b-1-b \, C^+(s)}{u-U^+(s)} + 
\frac{b-1-b \, C^-(s)}{u-U^-(s)}
$$
after standard algebra and the use of definition (\ref{C+-}) for constants 
$C^\pm(s)$ together with the relation $C^+(s) + C^-(s) = 1$, the homogeneous differential equation $u P(s,u) \partial_u F(s,u) + Q_b(s,u) F(s,u) = 0$ associated with (\ref{FbL}) has the general solution $F(s,\cdot)$ given by 
$F(s,u) = K \times u^{-b}
(u-U^+(s))^{b-1-b \, C^+(s)}(u-U^-(s))^{b-1-b \, C^-(s)}$ 
for any multiplicative constant $K$; using the method of the variation of constant $K$, the general solution to the full equation (\ref{FbL}) is easily derived as
\begin{align}
F_b^*(s,u) = & \, K_0 \times 
u^{-b}(u-U^+(s))^{b-1-b \, C^+(s)}(u-U^-(s))^{b-1-b \, C^-(s)} \; + 
\nonumber \\
& \, \frac{1}{u^b \, P(s,u)} 
\int_u^{U^-(s)} \Bigl[ \frac{z + b(1-z)}{(1-z)^2} + b \, E_{b-1}^*(z) \; + 
\nonumber \\
& \qquad \qquad \qquad 
([z - b(1+\varrho+s)]E_b^*(s,\QI) \Bigr] \mathfrak{R}(s,u;z)^{b}z^{b-1} \, \mathrm{d}z
\nonumber
\end{align}
for all $u \in \mathbb{D}$ and some constant $K_0$. Now, the analyticity of this solution $F_b^*$ at point $u = U^-(s) \in \mathbb{D}$ requires that this constant $K_0$ be zero, and expression (\ref{Fb*}) then follows. 
\end{proof}

\textbf{B)} Using Lemma \ref{ExpressFb*}, we now turn to the proof of Proposition \ref{CNCL}. The respective definitions (\ref{defFTheta}) and 
(\ref{defFbTheta}) of $F^*$ and $F_b^*$, $b \geq 1$, imply that $F^*$ is analytic at $u = 0$ if and only if all functions $F_b^*$, $b \geq 1$, are analytic at $u = 0$. Now, expression (\ref{Fb*}) shows that the pole at 
$u = 0$ is a false singularity if and only if the integral vanishes at 
$u = 0$, which translates to 
\begin{align}
& \int_0^{U^-(s)} \left[ \frac{z + b(1-z)}{(1-z)^2} + b \, E_{b-1}^*(z) \right ] \mathfrak{R}(s,0;z)^{b}z^{b-1} \, \mathrm{d}z \; = 
\nonumber \\ 
& - \; E_b^*(s,\QI) \int_0^{U^-(s)} [z - b(1+\varrho+s)] 
\mathfrak{R}(s,0;z)^{b}z^{b-1} \, \mathrm{d}z, \qquad b \geq 1.
\nonumber
\end{align}
By means of (\ref{defFbTheta}), $E_{b-1}^*$ can be expressed in terms of 
$F_{b-1}^*$ so that the latter condition equivalently reads
\begin{align}
\int_0^{U^-(s)} \left[ \frac{z + b(1-z)}{(1-z)^2} + b \, (z-\QI) F_{b-1}^*(z) 
+ b \, E_{b-1}^*(s,\QI) \right ] 
\mathfrak{R}(s,0;z)^{b}z^{b-1} \, \mathrm{d}z \; = 
\nonumber \\ 
- \; E_b^*(s,\QI) \int_0^{U^-(s)} [z - b(1+\varrho+s)] 
\mathfrak{R}(s,0;z)^{b}z^{b-1} \, \mathrm{d}z, \qquad b \geq 1.
\label{ConditionFb}
\end{align}
To gather this infinite set of conditions, multiply each side of equation 
(\ref{ConditionFb}) by $v^b/b!$, $v \in \mathbb{C}$, and sum all the obtained equalities over index $b \geq 1$; using the definitions
$E^*(s,\QI,v) = \sum_{b \geq 1} E_b^*(s,\QI) v^b/b!$ and 
$F^*(s,u,v) = \sum_{b \geq 1} F_b^*(s,u) v^b/b!$ of generating functions $E^*$ and $F^*$,  
we then get
\begin{align}
& \int_0^{U^-(s)} \Bigl[ \frac{e^{z \, \mathfrak{R}(s,0;z) \, v}-1}{(1-z)^2} + 
\frac{v}{1-z} \, \mathfrak{R}(s,0;z) e^{z \, \mathfrak{R}(s,0;z) \, v} \; + (z-\QI) \mathfrak{R}(s,0;z) \, v \, \times 
\nonumber \\
& F^*(s,z, z \, \mathfrak{R}(s,0;z) \, v) + \mathfrak{R}(s,0;z) \, v \, E^*(s,\QI,z \, \mathfrak{R}(s,0;z) \, v) \Bigr] \mathrm{d}z = \, -
\nonumber \\ 
& \int_0^{U^-(s)} \Bigl[ E^*(s,\QI,z \, \mathfrak{R}(s,0;z) \, v) - (1+\varrho+s)\mathfrak{R}(s,0;z) \, v \, \frac{\partial E^*}{\partial v}(s,\QI,z \, \mathfrak{R}(s,0;z) \, v) \Bigr] \mathrm{d}z
\nonumber
\end{align}
for all $v \in \mathbb{C}$. Now using the explicit definition (\ref{defL}) of function $L$ in terms of $E^*(s,\QI,\cdot)$ and the first order derivative 
$\partial_v E^*(s,\QI,\cdot)$, the latter relation can be easily recast in the form
\begin{align}
& \int_0^{U^-(s)} L(s,z,z \, \mathfrak{R}(s,0;z) \, v) \, 
\frac{\mathrm{d}z}{z} = \; - 
\nonumber \\
& \int_0^{U^-(s)} (z-\QI) \mathfrak{R}(s,0;z) \, v 
F^*(s,z,z \, \mathfrak{R}(s,0;z) \, v) \, \mathrm{d}z, 
\qquad v \in \mathbb{C}.
\label{ConditionL}
\end{align}
At this stage, we can further invoke the integral expression (\ref{IntGen0}) of $F^*$ to express the term $F^*(s,z,z \, \mathfrak{R}(s,0;z) \, v)$ of the right-hand side of (\ref{ConditionL}) in terms of $L$, giving
\begin{align}
F^*(s,z,z \, \mathfrak{R}(s,0;z) \, v) = \; & 
\frac{e^{\mathfrak{R}(s,0;z) \, v}}{P(s,z)} \; \times 
\nonumber \\
& \int_{z}^{U^-(s)} L(s,\zeta,\zeta \, \mathfrak{R}(s,0;z) \, v \, 
\mathfrak{R}(s,z;\zeta)) e^{-\mathfrak{R}(s,0;z) \, v \, 
\mathfrak{R}(s,z;\zeta)} \frac{\mathrm{d}\zeta}{\zeta} 
\nonumber \\
= \; & \frac{e^{\mathfrak{R}(s,0;z) \, v}}{P(s,z)} 
\int_{z}^{U^-(s)} L(s,\zeta,\zeta \, \mathfrak{R}(s,0;\zeta) \, v) 
e^{-\mathfrak{R}(s,0;\zeta) \, v} \, \frac{\mathrm{d}\zeta}{\zeta}
\label{ConditionL1}
\end{align}
after noting that 
$\mathfrak{R}(s,0;z)\mathfrak{R}(s,z;\zeta) = \mathfrak{R}(s,0;\zeta)$. As a consequence of (\ref{ConditionL1}), the right-hand side of (\ref{ConditionL}) now reads in the form
\begin{align}
& \int_0^{U^-(s)} (z-\QI) \, \mathfrak{R}(s,0;z) \, v \, 
F^*(s,z,z \, \mathfrak{R}(s,0;z) \, v) \, \mathrm{d}z \; =
\nonumber \\
& \int_0^{U^-(s)} (z-\QI) \mathfrak{R}(s,0;z) v 
\left ( \frac{e^{\mathfrak{R}(s,0;z) v}}{P(s,z)} 
\int_{z}^{U^-(s)} L(s,\zeta,\zeta \mathfrak{R}(s,0;\zeta) v) 
e^{-\mathfrak{R}(s,0;\zeta) v} \frac{\mathrm{d}\zeta}{\zeta} \right ) 
\mathrm{d}z;
\nonumber
\end{align}
interchanging the order of integration in the right-hand side of the latter equality, we obtain
\begin{align}
& \int_0^{U^-(s)} (z-\QI) \, \mathfrak{R}(s,0;z) \, v \, 
F^*(s,z,z \, \mathfrak{R}(s,0;z) \, v) \, \mathrm{d}z \; =
\nonumber \\
& v \, \int_0^{U^-(s)} H(s,\zeta,v) \, 
L(s,\zeta,\zeta \mathfrak{R}(s,0;\zeta) v) \, \frac{\mathrm{d}\zeta}{\zeta}
\label{ConditionL2}
\end{align}
where we set
$$
H(s,\zeta,v) = e^{-\mathfrak{R}(s,0;\zeta) \, v} 
\int_0^\zeta (z-\QI) \, \frac{\mathfrak{R}(s,0;z)}{P(s,z)} \,  
e^{\mathfrak{R}(s,0;z) \, v} \, \mathrm{d}z.
$$
After (\ref{ConditionL2}), relation (\ref{ConditionL}) can therefore be written in the form
\begin{align}
& \int_0^{U^-(s)} L(s,z,z \, \mathfrak{R}(s,0;z) \, v) \, 
\frac{\mathrm{d}z}{z} \; = 
\nonumber \\
& - v \, \int_0^{U^-(s)} H(s,\zeta,v) \, 
L(s,\zeta,\zeta \, \mathfrak{R}(s,0;\zeta) \, v) \, 
\frac{\mathrm{d}\zeta}{\zeta}, \qquad v \in \mathbb{C}.
\label{ConditionL3}
\end{align}
The kernel $H$ introduced in (\ref{ConditionL2}) can be actually explicitly calculated. In fact, expression (\ref{defR}) yields
\begin{equation}
\frac{\mathrm{d}}{\mathrm{d}u}\mathfrak{R}(s,u_0;u) = 
\frac{\mathfrak{R}(s,u_0;u)}{(u-U^-(s))(u-U^+(s))} (\QI - u)
\label{DerivR}
\end{equation}
(where we have used the identity
$-(C^-(s)-1)U^+(s)-(C^+(s)-1)U^-(s) = \QI 
$ 
easily derived from the definition (\ref{C+-}) of exponents $C^+(s)$ and 
$C^-(s)$); it then follows from (\ref{DerivR}) that
$e^{\mathfrak{R}(s,0;z) \, v} \, v \, 
\mathfrak{R}(s,0;z)(\QI - z)/P(s,z) = 
\mathrm{d} \left [ e^{\mathfrak{R}(s,0;z) \, v}\right ] /\mathrm{d}z$ 
is an exact derivative, hence
$$
H(s,\zeta,v) = e^{-\mathfrak{R}(s,0;\zeta) \, v}  \times 
\left [ - \frac{e^{\mathfrak{R}(s,0;z) \, v}}{v} \right ]_{z = 0}^{z = \zeta} 
= - \, \frac{1-e^{-\mathfrak{R}(s,0;\zeta) \, v}}{v}.
$$
Substituting this expression of $H(s,\zeta,v)$ in the right-hand side of 
(\ref{ConditionL3}), the latter readily reduces to condition (\ref{CNC0}) on function $L$, as claimed $\blacksquare$

\end{document}